\documentclass[12pt]{amsart}
\usepackage[latin1]{inputenc}
\usepackage{color}
\usepackage{enumerate}
\usepackage{amssymb}
\usepackage[all]{xy}
\usepackage{hyperref}

\newtheorem{thm}{Theorem}[section]
\newtheorem{cor}[thm]{Corollary}
\newtheorem{lem}[thm]{Lemma}
\newtheorem{con}[thm]{Conjecture}
\newtheorem{prop}[thm]{Proposition}
\theoremstyle{definition}
\newtheorem{defn}[thm]{Definition}

\theoremstyle{remark}
\newtheorem{rem}[thm]{Remark}

\newtheorem{example}{Example}
\numberwithin{equation}{section}
\newcommand{\norm}[1]{\left\Vert#1\right\Vert}
\newcommand{\abs}[1]{\left\vert#1\right\vert}
\newcommand{\set}[1]{\left\{#1\right\}}

\newcommand{\To}{\longrightarrow}

\newcommand{\ds}{\displaystyle}
\newcommand{\R}{\mathbb{R}}
\newcommand{\N}{\mathbb{N}}

\DeclareMathOperator{\NA}{NA}

\DeclareMathOperator{\Hom}{Hom}

\usepackage{color}
\usepackage{enumerate}
\begin{document}
\setcounter{tocdepth}{1}


\title{Norm-attaining lattice homomorphisms}
\author[Dantas]{Sheldon Dantas}
\address[Dantas]{Department of Mathematics, Faculty of Electrical Engineering, Czech Technical University in Prague, Technick\'a 2, 166 27, Prague 6, Czech Republic \newline
	\href{http://orcid.org/0000-0001-8117-3760}{ORCID: \texttt{0000-0001-8117-3760} } }
\email{\texttt{gildashe@fel.cvut.cz}}

\author[Mart\'inez-Cervantes]{Gonzalo Mart\'inez-Cervantes}
\address[Mart\'inez-Cervantes]{Universidad de Murcia, Departamento de Matem\'{a}ticas, Campus de Espinardo 30100 Murcia, Spain
	\newline
	\href{http://orcid.org/0000-0002-5927-5215}{ORCID: \texttt{0000-0002-5927-5215} } }	

\email{gonzalo.martinez2@um.es}

\author[Rodr\'iguez Abell\'an]{Jos\'e David Rodr\'iguez Abell\'an}
\address[Rodr\'iguez Abell\'an]{Universidad de Murcia, Departamento de Matem\'{a}ticas, Campus de Espinardo 30100 Murcia, Spain 	\newline
	\href{https://orcid.org/0000-0002-2764-0070}{ORCID: \texttt{0000-0002-2764-0070} }}

\email{josedavid.rodriguez@um.es}

\author[Rueda Zoca]{Abraham Rueda Zoca}
\address[Rueda Zoca]{Departamento de An\'alisis Matem\'atico, 18071, Granada, Spain
	\newline
	\href{https://orcid.org/0000-0003-0718-1353}{ORCID: \texttt{0000-0003-0718-1353} }}
\email{\texttt{abrahamrueda@ugr.es}}
\urladdr{\url{https://arzenglish.wordpress.com}}

\thanks{S. Dantas was supported by the project OPVVV CAAS CZ.02.1.01/0.0/0.0/16\_019/0000778 and by the Estonian Research Council grant PRG877. G. Mart\'inez-Cervantes and J. D. Rodr\'iguez Abell\'an were supported by the project MTM2017-86182-P (Government of Spain, AEI/FEDER, EU) and project 20797/PI/18 by Fundaci\'{o}n S\'{e}neca, ACyT Regi\'{o}n de Murcia. The research of G. Mart\'inez-Cervantes has been co-financed by the European Social Fund (ESF) and the Youth European Initiative (YEI) under the Spanish Seneca Foundation (CARM) (ref. 21319/PDGI/19).
	J. D. Rodr\'iguez Abell\'an was supported by FPI contract of Fundaci\'on S\'eneca, ACyT Regi\'{o}n de Murcia.
	The research of A. Rueda Zoca was supported by MICINN (Spain) Grant PGC2018-093794-B-I00 (MCIU, AEI, FEDER, UE), by Junta de Andaluc\'ia Grant A-FQM-484-UGR18 and by Junta de Andaluc\'ia Grant FQM-0185}



\begin{abstract} 
In this paper we study the structure of the set $\Hom(X,\R)$ of all lattice homomorphisms from a Banach lattice $X$ into $\R$. Using the relation among lattice homomorphisms and disjoint families, we prove that the topological dual of the free Banach lattice $FBL(A)$ generated by a set $A$ contains a disjoint family of cardinality $2^{|A|}$, answering a question of B. de Pagter and A. W. Wickstead. We also deal with norm-attaining lattice homomorphisms.
For classical Banach lattices, as $c_0$, $L_p$-, and $C(K)$-spaces, every lattice homomorphism on it attains its norm, which shows, in particular, that there is no James theorem for this class of functions. We prove that, indeed, every lattice homomorphism on $X$ and $C(K,X)$ attains its norm whenever $X$ has order continuous norm.
On the other hand, we provide what seems to be the first example in the literature of a lattice homomorphism which does not attain its norm. In general, we study the existence and characterization of lattice homomorphisms not attaining their norm in free Banach lattices. 
As a consequence, it is shown that no Bishop-Phelps type theorem holds true in the Banach lattice setting, i.e. not every lattice homomorphism can be approximated by norm-attaining lattice homomorphisms.
\end{abstract}

\maketitle


\section{Introduction} 
	
It is well-known that in a Banach space $E$, the set of all continuous linear functionals from $E$ into $\R$ determines almost totally the structure of $E$ as a Banach space. As a matter of fact, by the Hahn-Banach Theorem, the norm of any element $x \in E$ can be calculated as the supremum over all continuous functionals in the unit ball $B_{E^*}$ of the topological dual space $E^*$ of $E$. On the other hand, James theorem \cite{J} states that a Banach space is reflexive if and only if every functional in the dual attains its norm. Moreover, Bishop and Phelps \cite{BP} proved that every functional can be approximated by functionals which attain their norms.
In the Banach lattice setting, it is natural to consider that the role of linear continuous functionals is played by lattice homomorphisms, i.e. the linear continuous functionals which, in addition, respect lattice operations. In this paper, we wonder 
what can be deduced about a Banach lattice from its set of lattice homomorphisms.

\bigskip

Since James and Bishop-Phelps theorems, the theory of norm-attaining functionals were intensively studied. In fact, this theory has been widely extended to different contexts besides linear functionals. Indeed, among others, some authors considered it in the context of linear operators (see \cite{B, H, JW, linds2, S, U, Z}); others studied norm-attaining bilinear mappings (see \cite{gasp, AFW, Choi}); and more recently several problems on norm-attainment of homogeneous polynomials and Lipschitz maps were considered (see  \cite{AGM, ADM} and \cite{CCGMR,MPR2,CGMR}, respectively). In the context of homomorphisms on Banach lattices, we should highlight the recent paper \cite{OT}, where a James type theorem was proved for positive linear functionals on some Banach lattices (see \cite[\textsection 6]{OT}). Notice that positive linear functionals are functionals which respect the order in the Banach lattice, but they do not need to respect the lattice operations; in this paper we focus on the much more restrictive subclass of the set of positive linear functionals $x^*$ on $X^*$ for which, moreover, both equalities $x^*(x \vee y)=x^* (x) \vee x^* (y)$ and $x^*(x \wedge y)=x^* (x) \wedge x^*(y)$ hold for every $x,y\in X$ (i.e. the subclass of lattice homomorphisms). Whereas it is simple to provide an example of a positive linear functional not attaining its norm in a Banach lattice, finding a not norm-attaining lattice homomorphism becomes a delicate problem. Indeed, as far as we know, we present in this paper the first examples of such elements.

\bigskip

Let us describe the content of this paper. Section \ref{section:notation} is devoted to present some notation and necessary background. We will be working with the free Banach lattices $FBL(A)$ generated by a set $A$ with no extra structure as well as the (more general) Banach lattices $FBL[E]$ generated by a Banach space $E$. We prove some elementary results that will be useful to solve some problems throughout the article. 

In Section \ref{SectionStructureHom}, we provide some general results on the structure of $\Hom(X, \R)$ and its relation with disjoint families, which allow us to answer a question posed by B. de Pagter and A. W. Wickstead in \cite{dPW15}. Moreover, we show that every separable Banach lattice embeds into a Banach lattice whose set of lattice homomorphisms is trivial, i.e. a Banach lattice $X$ for which $\Hom(X, \R)=\{0\}$.

In Section \ref{SectionPositiveCases}, motivated by \cite{OT}, we wonder whether there is a James type theorem for $\Hom(X, \R)$. For classical Banach lattices (as $c_0$, $L_p(\mu)$-, and $C(K)$-spaces), the set $\Hom(X, \R)$ is very small, in the sense that, not just a James type theorem does not hold, but also that every homomorphism attains its norm. 
Along this section we prove that every lattice homomorphism on $X$ and $C(K,X)$ attains its norm whenever $X$ has order continuous norm.

In Section \ref{SectionNegativeCases}, using free Banach lattices we are able to present the first examples of lattice homomorphisms which do not attain their norm. In particular, we show that if $E$ is an $L_1$-space, a separable $L_1$-predual or a Lipschitz-free Banach space over a metric space with cluster points, then $\Hom(FBL[E], \R)$ contains a lattice homomorphism which does not attain its norm.
Moreover, we characterize lattice homomorphisms attaining their norm on $FBL[E]$ whenever $E$ is an isometric predual of $\ell_1(A)$ or is isometric to $\ell_1(A)$ for some infinite set $A$. These results allow us to show that no Bishop-Phelps theorem holds in the class of Banach lattices, i.e. that there are lattice homomorphisms which cannot be approximated in norm by  norm-attaining lattice homomorphisms.

\section{Background and Notation} \label{section:notation}

Let us present all the necessary background material so that the paper can be fully accessible. Throughout the paper, all the Banach spaces and Banach lattices are considered to be {\bf real}. If $X, Y$ are Banach lattices, we say that $T\colon X \longrightarrow Y$ is a {\it Banach lattice homomorphism}, or simply, a {\it lattice homomorphism}, if it is a linear bounded operator preserving the lattice operations, that is, $T(x \wedge y) = T(x) \wedge T(y)$ and $T(x \vee y) = T(x) \vee T(y)$ for every $x, y \in X$. By a \textit{lattice homomorphism on} $X$ we mean a functional in $X^*$ which also preserves suprema and infima. We denote by $B_X$ the unit ball of $X$ and by $S_X$ the unit sphere of $X$.

Given a non-empty set $A$ with no extra structure, the {\it free Banach lattice generated by the set $A$} is a Banach lattice $F$ together with a bounded map $\phi\colon A \longrightarrow F$ with the property that for every Banach lattice $X$ and every bounded map $T \colon A \longrightarrow X$, there is a unique Banach lattice homomorphism $\hat{T}\colon F \longrightarrow X$ such that $T = \hat{T} \circ \phi$ and $\|{\hat{T}}\| = \| T \|$. In other words, the following diagram commutes:
$$\xymatrix{A\ar_{\phi}[d]\ar[rr]^T&&X\\
	F\ar_{\hat{T}}[urr]&& }$$
Let us clarify here that the norm of $T$ is given by $\sup \set{\norm{T(a)} : a \in A}$ while the norm of $\hat{T}$ is the usual one for Banach spaces. We refer the reader to the seminal paper \cite{dPW15} for more background on free Banach lattices generated by a set. On the other hand, the concept of a Banach lattice freely generated by a given Banach space $E$ was recently introduced  by A. Avil\'es, J. Rodr\'iguez, and P. Tradacete in \cite{ART18}. This provides a new tool for better understanding the relation between Banach spaces and Banach lattices.  The {\it free Banach lattice generated by a Banach space $E$} is a Banach lattice $F$ together with a bounded operator $\phi\colon E \longrightarrow F$ with the property that for every Banach lattice $X$ and every bounded operator $T \colon E \longrightarrow X$, there is a unique Banach lattice homomorphism $\hat{T}\colon F \longrightarrow X$ such that $T = \hat{T} \circ \phi$ and $\| \hat{T} \| = \| T \|$. In other words, the following diagram commutes:
$$\xymatrix{E\ar_{\phi}[d]\ar[rr]^T&&X\\
	F\ar_{\hat{T}}[urr]&& }$$

\noindent
This definition generalizes the notion of a free Banach lattice generated by a set $A$. Indeed, the free Banach lattice generated by a set $A$ can be naturally identified with the free Banach lattice generated by the Banach space $\ell_1(A)$ (see Corollary 2.8 in \cite{ART18}).

\bigskip

It is possible, though, to give an explicit description of the free Banach lattice $FBL(A)$ as a space of functions. Indeed, for $a \in A$, let $\delta_a\colon[-1,1]^A \longrightarrow \mathbb{R}$ be the {\it evaluation function} given by $\delta_a(x^*) = x^*(a)$ for every $x^* \in [-1,1]^A$. For $f\colon[-1,1]^A \longrightarrow \mathbb{R}$, define the norm
\begin{equation*} 
\|f\|_{FBL(A)} = \sup \set{\sum_{i = 1}^n \abs{ f(x_{i}^{\ast})} :  n \in \mathbb{N}, \, x_1^{\ast}, \ldots, x_n^{\ast} \in [-1,1]^A, \text{ }\sup_{a \in A} \sum_{i=1}^n \abs{x_i^{\ast}(a)} \leq 1 }.
\end{equation*} 
Then, the free Banach lattice generated by $A$ is the Banach lattice generated by $\{\delta_a : a\in A\}$ inside the Banach lattice of the functions in $\mathbb{R}^{[-1,1]^A}$ with finite norm $\| \cdot \|_{FBL(A)}$, endowed with the pointwise order  and the pointwise operations. The natural identification of $A$ in $FBL(A)$ is given by the map $\phi\colon A \longrightarrow FBL(A)$ defined by $\phi(a) = \delta_a$ for every $a \in A$. Since every function in $FBL(A)$ is a uniform limit of such functions, they are all continuous (in the product topology) and positively homogeneous (i.e. $f(\lambda x^*)=\lambda f(x^*)$ for every $x^* \in [-1,1]^A$ and every $0\leq \lambda\leq 1$).

Analogously, for the free Banach lattice $FBL[E]$, it is also possible to give an explicit description of it. For $x \in E$, let $\delta_x\colon E^* \longrightarrow \mathbb{R}$ be the {\it evaluation function} given by $\delta_x(x^*) = x^*(x)$ for every $x^* \in E^*$. For a function $f\colon E^\ast \To \mathbb{R}$, define the norm 
\begin{equation*} 
\norm{f}_{FBL[E]} = \sup \set{\sum_{i = 1}^n \abs{f(x_{i}^{*})} : n \in \mathbb{N}, \, x_1^{*}, \ldots, x_n^{*} \in E^{*},\text{ }\sup_{x \in B_E} \sum_{i=1}^n \abs{x_i^{*}(x)} \leq 1 }.
\end{equation*}

\noindent
Then, the free Banach lattice generated by $E$ is the Banach lattice generated by $\{\delta_x : x\in E\}$ inside the Banach lattice of the functions in $\mathbb{R}^{E^*}$ with finite norm $\| \cdot \|_{FBL[E]}$, endowed with the pointwise order and the pointwise operations. The natural identification of $E$ in $FBL[E]$ is given by the map $\phi\colon E \longrightarrow FBL[E]$ defined by $\phi(x) = \delta_x$ for every $x \in E$ (let us notice that it is a linear isometry between $E$ and its image in $FBL[E]$). Moreover, all the functions in $FBL[E]$ are positively homogeneous  (i.e. $f(\lambda x^*)=\lambda f(x^*)$ for every $x^* \in E^*$ and every $0\leq \lambda$) and $w^*$-continuous when restricted to the closed unit ball of  $B_{E^\ast}$ (see \cite[Lemma 4.10]{ART18}). 

\bigskip 
 
We will need the following definition.
 
\begin{defn}
	Let $E$ be a Banach space and $f \colon E^* \longrightarrow \R$ be a function in $FBL[E]$. We will say that \textit{$f$ depends on finitely many coordinates} if there exists a finite subset $E_0 \subseteq E$ such that $f(x^*) = f(y^*)$ whenever $x^*\vert_{E_0} = y^*\vert_{E_0}$.
\end{defn}

Let us notice that each $\delta_x$ depends only on one coordinate, namely the element $x$ itself. Since every function in $FBL[E]$ can be approximated by a finite lattice linear combination of $\delta_{x}$'s, we can highlight the following remark.

\begin{rem} \label{LemFunctionsDependFinitelyManyCoordinates} Every function in $FBL[E]$ can be approximated by a function which depends on finitely many coordinates. Consequently, every function in $FBL[E]$ depends on countably many coordinates.
\end{rem}

We will be working with the set of all lattice homomorphisms on a Banach lattice $X$ denoted by $\Hom(X,\R)$. Let us notice that this subset is {\it not} a linear subspace of $X^*$; indeed, in general, the sum of two lattice homomorphisms is no longer a lattice homomorphism, which gives us a big difference between the category of Banach lattices and lattice homomorphisms and the category of Banach spaces and linear functionals. Moreover, the set $\Hom(X,\R)$ is a $w^*$-closed subset of $X^*$. If $E$ is a normed space, we say that $x^* \in E^*$ {\it attains its norm} or it is {\it norm-attaining}, if there is $x_0 \in S_E$ such that $|x^*(x_0)| = \|x^*\| = \sup_{x \in S_E} |x^*(x)|$. We denote by $\NA(E, \R)$ the set of all norm-attaining functionals on $E^*$.

\bigskip

Let us finish this section by presenting some basic definitions and results on almost isometric ideals in Banach spaces. We will be using these tools intensively in Section \ref{SectionNegativeCases}. Let $E$ be a Banach space. A subspace $Z$ of $E$ is said to be an {\it almost isometric ideal} (\emph{ai-ideal}, for short) in $E$ if for each $\varepsilon>0$ and for each
finite-dimensional subspace $F\subseteq E$ there exists a linear
operator $T\colon F\longrightarrow Z$ satisfying
\begin{enumerate}
	\item\label{item:ai-1}
	$T(x)=x$ for each $x\in F\cap Z$, and
	\item\label{item:ai-2}
	$(1-\varepsilon) \Vert x \Vert \leq \Vert T(x)\Vert\leq
	(1+\varepsilon) \Vert x \Vert$
	for each $x\in F$.
\end{enumerate}
If $T$ satisfies only (\ref{item:ai-1}) and the right-hand side of
(\ref{item:ai-2}), we say that $Z$ is an \emph{ideal} in $E$ (see \cite{gks}). Let us notice that, in the context of almost isometric ideals, the Principle of Local Reflexivity means exactly that $E$ is an ai-ideal in $E^{**}$
for every Banach space $E$. We will need the following result.

\begin{thm}[\mbox{\cite[Theorem 1.4]{aln2}}]\label{theo:hbaioper}
	Let $E$ be a Banach space and $Z$ an almost isometric ideal in $E$. Then, there is a linear isometry $\varphi\colon  Z^*\longrightarrow E^*$ such that
	$$\varphi(z^*)(z)=z^*(z)$$
	holds for every $z\in Z$ and $z^*\in Z^*$ and satisfying that, for every $\varepsilon>0$, every finite-dimensional subspace $F_0$ of $E$ and every finite-dimensional subspace $F_1$ of $Z^*$, we can find an operator $T\colon F_0\longrightarrow Z$ satisfying
	\begin{enumerate}
		\item $T(x)=x$ for every $x\in F_0\cap Z$, 
		\item $(1-\varepsilon)\Vert x\Vert\leq \Vert T(x)\Vert\leq (1+\varepsilon)\Vert x\Vert$ for every $x\in F_0$, and
		\item $f(T(x))=\varphi(f)(x)$ for every $x\in F_0$ and every $f\in F_1$.
	\end{enumerate}
\end{thm}

\noindent
Following the terminology of \cite{abrahamsen}, the isometry $\varphi$ is called an \textit{almost isometric Hahn-Banach extension operator}. Notice that if $\varphi\colon Z^*\longrightarrow E^*$ is an almost isometric Hahn-Banach extension operator, then $\varphi^*\colon E^{**}\longrightarrow Z^{**}$ is a norm-one projection (see e.g. \cite[Theorem 3.5]{kaltonlocomp}). Finally, we use the following theorem, whose proof follows the lines of the main theorem of \cite{SY}, in the proof of Theorem \ref{falloconjec}.

\begin{thm}[\mbox{\cite[Theorem 1.5]{abrahamsen}}]\label{theo:exteaiideales}
	Let $E$ be a Banach space, $Y \subseteq E$ a separable subspace of $E$, and $W\subseteq E^*$ a separable subspace of $E^*$. Then there exist a separable almost isometric ideal $Z$ in $E$ containing $Y$ and an almost isometric Hahn-Banach extension operator $\varphi\colon Z^*\longrightarrow E^*$ such that $\varphi(Z^*)\supset W$.
\end{thm}

\section{The structure of $\Hom(X, \R)$ and disjoint families}\label{SectionStructureHom}

We start this section by giving some structural results on the set $\Hom(X, \R)$, where $X$ is a Banach lattice. In particular, we focus on the relation between lattice homomorphisms and disjoint families. This relation will appear in a natural way through the concept of atoms.
In particular, we will have that linearly independent lattice homomorphisms are disjoint. One of the consequences of this fact will be the failure of the lattice analogous of Bishop-Phelps theorem (see Theorem \ref{TheoremNoBishop-Phelps} on Section \ref{SectionNegativeCases}).

\bigskip

For an element $x^*$ in the Banach lattice $X^*$, it is worth mentioning, although straightforward, that in general we have that $x^*(x \vee y) \not= x^*(x) \vee x^*(y)$. For example, on the Banach lattice $c_0$ with its natural order structure, we have that
\begin{equation*}
	(e_1^* + e_2^*)(e_1 \vee e_2) = 2 \not= 1 = (e_1^* + e_2^*)(e_1) \vee (e_1^* + e_2^*)(e_2),
\end{equation*}
where $(e_i, e_i^*)$ is the biorthogonal system of $c_0$. Analogously, it is possible to show that $x^*(x \wedge  y) = x^*(x) \wedge  x^*(y)$ does not hold in general. Indeed, it follows from the Riesz-Kantorovich formulae that
$$  x^*(x \wedge  y)=\inf\{y^*(x)+(x^*-y^*)(y): 0 \leq y^* \leq x^*\}$$
and
$$  x^*(x \vee  y)=\sup\{y^*(x)+(x^*-y^*)(y): 0 \leq y^* \leq x^*\},$$
whenever $x^*$ is a positive element in $X^*$.

As we have already mentioned in the previous section, we will be interested in the set $\Hom(X, \R)$, where the identities $x^*(x \vee y) = x^*(x) \vee x^*(y)$ and $x^*(x \wedge y) = x^*(x) \wedge x^*(y)$ hold to be true. We refer the reader to \cite[Section 1.3]{Meyer} for a detailed background on lattice homomorphisms.

It is clear that any lattice homomorphism $x^*$ on a Banach lattice $X$ is positive (i.e. $x^*(x) \geq 0$ for every positive $x \in X$). In fact, an element $x^* \in X^*$ with $x^*>0$ (i.e. $x^*$ is positive and $x^* \neq 0$) is a lattice homomorphism if and only if $x^*$ is an atom in $X^*$ (see \cite[Section 2.3, Exercise 6]{AA}). Recall that an element $x>0$ in a Banach lattice $X$ is an \textit{atom} if and only if $x\geq u \geq 0$ implies that $u=ax$ for some scalar $a\geq 0$. Due to this characterization, we have the following result, which will be used in Proposition \ref{theorem:disjointfamily1} later on. Let us recall that $x$ and $y$ are said to be {\it disjoint} whenever $|x| \wedge |y| = 0$.
 
\begin{lem}
\label{LemmDisjointOrDependent}
Let $X$ be a Banach lattice and $x^*,y^* \in \Hom(X,\R)$. Then, $x^*$ and $y^*$ are either linearly dependent or disjoint.
\end{lem}
\begin{proof}
Since $x^*$ and $y^*$ are lattice homomorphisms, we have $x^* ,y^*\geq 0$. We can suppose that $x^*$ and $y^*$ are nonnull. Set $z^*=|x^*|\wedge|y^*|=x^*\wedge y^*$. Since $0 \leq z^* \leq x^*$ and $x^*$ is an atom, we have that there exists $a_1 \in \R$ such that $z^* = a_1 x^*$. Analogously, there exists $a_2 \in \R$ such that $z^* = a_2 y^*$. Thus, $|x^*|\wedge|y^*|=a_1 x^* = a_2 y^*$ and the conclusion follows.
\end{proof}

\begin{cor}
\label{CorDistanceHomomorphisms}	
Let $X$ be a Banach lattice and $x^*,y^*\in \Hom(X,\R)$ be linearly independent. Then $\|x^*-y^*\| \geq \max\{\|x^*\|,\|y^*\|\}$.
\end{cor}
\begin{proof}
By the previous lemma, both elements are disjoint (and positive). In particular, the positive part of $x^*-y^*$ is $x^*$ and the negative part is $y^*$. Thus, we have that $|x^*-y^*|=x^*+y^*$ and, therefore,
$$ \|x^*-y^*\| = \| |x^*-y^*|\|= \| x^*+y^* \| \geq \max\{\|x^*\|,\|y^*\|\},$$
where in the last inequality we have used that both $x^*$ and $y^*$ are smaller than $x^*+y^*$. 
\end{proof}

\bigskip

For classical Banach lattices $X$ with their usual order and norm, we have that $\Hom(X, \R)$ is very small as described in Example \ref{ExampleSetHomClassical}. In what follows, $\delta_x \colon C(K) \longrightarrow \R$ is the evaluation function on $x$ in a $C(K)$-space. Item (i) can be easily computed using the equivalence between nonnull lattice homomorphisms and atoms, item (ii) is proved in \cite[Lemma 4.23]{AA}, and item (iii) follows from the fact that the dual of every atomless Banach lattice with order continuous norm is also atomless (see \cite[Lemma 2.31]{AA}). 

\begin{example}\label{ExampleSetHomClassical} Let $X$ be $c_0$ or $\ell_p$, $K$ a compact Hausdorff topological space, and $1 \leq p < \infty$. 
\begin{itemize}
	\item[(i)] $\Hom(X, \R) = \{ \lambda e_n^*: \lambda \geq 0, n \in \N \}$.
	\vspace{0.1cm}
	\item[(ii)] $\Hom(C(K), \R) = \{ \lambda \delta_x: \lambda \geq 0, x \in K \}$.
	\vspace{0.1cm}
	\item[(iii)] $\Hom(L_p[0,1], \R) = \{0\}$.
\end{itemize}
\end{example}

 Let $K$ be a compact Hausdorff topological space and $X$ a Banach lattice. The Banach space of all continuous functions from $K$ into $X$, denoted by $C(K,X)$, is a Banach lattice when endowed with the pointwise order, that is, $f \leq g$ if and only if $f(x) \leq g(x)$ for every $x \in K$. These spaces play an important role in the theory of Banach lattices. Notice that, although every separable Banach space embeds into $C([0,1])$, this result is no longer true when we restrict to the class of separable Banach lattices and lattice embeddings. Instead, it was proved in \cite{LLOT} that the Banach lattice $C(\Delta, L_1[0,1])$ is injectively universal for the class of separable Banach lattices, i.e.~any separable Banach lattice embeds lattice isometrically into $C(\Delta, L_1[0,1])$, where  $\Delta$ is the Cantor set. The following lemma is a consequence of \cite[Theorem 2.2]{CRX}.
 
 \begin{lem} \label{lemma:HomomC(K,X)}
 	Let $X$ be a Banach lattice and $K$ a compact Hausdorff topological space. Then, for every nonnull $\varphi\in\Hom(C(K,X),\R)$, there exist a unique $a \in K$ and $x^* \in \Hom(X,\R)$ such that $\varphi(f)=x^*(f(a))$ for every $f\in C(K,X)$.
 \end{lem}
 \begin{proof}
 	It follows from \cite[Theorem 2.2]{CRX} that there exists a unique $a\in K$ such that $\varphi(f)=\varphi(1 \otimes f(a))$ for every function $f\in C(K,X)$, where $1 \otimes f(a)$ denotes the constant function equal to $f(a)$. Notice that $\hat{X}=\{1 \otimes x \in  C(K,X) : x\in X\}$ is a sublattice isometric to $X$. Indeed, $i\colon X \longrightarrow \hat{X}$ defined by $i(x)=1 \otimes x$ is a lattice isometry. 
 	Thus,  $x^*:=\varphi|_{\hat{X}}\circ i \in \Hom(X, \R)$ and
 	$$\varphi(f)=\varphi(1 \otimes f(a))=\varphi(i(f(a)))=x^*(f(a))$$
 	for every  $f\in C(K,X)$.
 \end{proof}

 The next corollary shows the drastic failure of the lattice version of the Hahn-Banach Theorem; every separable Banach lattice $X$ can be embedded into a Banach lattice $Y$ in which no nonnull lattice homomorphism in $\Hom(X, \R)$ can be extended to a lattice homomorphism in $\Hom(Y, \R)$:
 
 \begin{cor} \label{Corhomozero} Let $X$ be a Banach lattice such that $\Hom(X, \R)=\{0\}$. Then, 
 \begin{equation*} 	
 \Hom(C(K,X),\R)=\{0\}.  
 \end{equation*} 
 In particular, every separable Banach lattice embeds lattice isometrically into a Banach lattice on which there are no nontrivial homomorphisms.
 \end{cor}
 \begin{proof}
 	The first part follows from the previous lemma, whereas the second part follows from the fact that $\Hom(L_1[0,1],\R)=\{0\}$ and \cite[Theorem 1.1]{LLOT}.
 \end{proof}

\bigskip

We finish this section considering free Banach lattices. Let $E$ be a Banach space. For $x^* \in E^*$, we denote by $\delta_{x^*} \colon FBL[E] \longrightarrow \R$ the {\it evaluation function} on $FBL[E]$ given by $\delta_{x^*}(f) = f(x^*)$ for every $f \in FBL[E]$. Analogously, if $A$ is a non-empty set and $x^* \in [-1,1]^A$, $\delta_{x^*} \colon FBL(A) \longrightarrow \R$ is the evaluation function on $FBL(A)$. 

\begin{prop} \label{PropSetHomFBL} Let $E$ be a Banach space and $A$ be a non-empty set.
\begin{itemize}
	\item[(i)] $\Hom(FBL(A), \R) = \{ \lambda\delta_{x^*}: \lambda \geq 0, x^* \in [-1,1]^A\}$ (see \cite[Theorem 5.5]{dPW15}).
		\item[(ii)] $\Hom(FBL[E], \R) = \{ \delta_{x^*}: x^* \in E^* \}$ (see \cite[Corollary 2.6]{ART18}).
\end{itemize}
\end{prop}

Let us now use Lemma \ref{LemmDisjointOrDependent} to deal with disjoint families in $FBL(A)^*$ and $FBL[E]^*$. Motivated by the study of free and projective objects, disjoint families in free Banach lattices were studied in \cite{dPW15} and, more recently, in \cite{APRA}. It was proved in \cite{dPW15} that disjoint families in $FBL(A)$ can only be at most countable (this was proved in a more general way in \cite{APRA}, where the authors showed that the free Banach lattice $FBL[E]$ satisfies the $\sigma$-bounded chain condition (see \cite[Theorem 1.2]{APRA})), although $FBL(A)^*$ always contains a disjoint family of cardinality $|A|$.
Question 12.8 in \cite{dPW15} asks how large disjoint families in $FBL(A)^*$ can be. Thanks to the advances made on the understanding of free Banach lattices during the past few years and the relation between lattice homomorphisms and disjoint families, we will easily show that, indeed, there are disjoint families of cardinality $2^{|A|}$.

Let $E$ be a Banach space. Then, it is immediate that $\delta_{x^*},\delta_{y^*} \in FBL[E]^*$ are linearly independent whenever $x^*,y^* \in E^*$ are linearly independent.
Moreover, if $0 \neq x^*=-ay^*$ with $a>0$, then $\delta_{x^*}$ and $\delta_{y^*}$ are also linearly independent, since both are nonzero and if $x \in E$ is any element for which $x^*(x)>0$, we have that $\delta_{x^*}( \delta_x \vee 0)=x^*(x) >0$ but 
\begin{equation*} 
\delta_{y^*}( \delta_x \vee 0)= y^*(x) \vee 0 = \left(-\frac{1}{a} x^*(x) \right) \vee 0 = 0.
\end{equation*} 
Thus, by Lemma \ref{LemmDisjointOrDependent}, the set $\{\delta_{x^*}: x^* \in S_{E^*}\}$ is a disjoint family of cardinality $|S_{E^*}|$ and the next proposition follows:
\begin{prop} \label{theorem:disjointfamily1}
Let $E$ be a Banach space. Then, $FBL[E]^*$ contains a disjoint family of cardinality $|S_{E^*}|$.
\end{prop}

We use Proposition \ref{theorem:disjointfamily1} to answer \cite[Question 12.8]{dPW15}.

\begin{thm}\label{TheoremDisjointFamiliesInTheBanachSpaceDualOfFBL(A)}
	If $A$ is an infinite set, then $FBL(A)^*$ contains a disjoint family of cardinality $2^{|A|}$. Moreover, there is no disjoint family of cardinality larger than $2^{|A|}$.
\end{thm}

\begin{proof} We have that $FBL(A)=FBL[\ell_1(A)]$ (see \cite[Corollary 2.8]{ART18}). By Proposition \ref{theorem:disjointfamily1}, $FBL(A)^*$ contains a disjoint family of cardinality $|S_{\ell_\infty(A)}|=2^{|A|}$. Let us prove now that $|FBL(A)^*| \leq 2^{|A|}$. Indeed, let $\hat{A}$ be the smallest subset of $FBL(A)$ containing $A$, and closed under the operations $\wedge$ and $\vee$ and finite linear combinations with coefficients in $\mathbb{Q}$. Let $R \colon FBL(A)^* \longrightarrow \mathbb{R}^{\hat{A}}$ be the restriction map given by $R(f^*) = f^* \vert_{\hat{A}}$ for every $f^* \in FBL(A)^*$. Since $\hat{A}$ is dense in $FBL(A)$, we have that $R$ is injective. Thus, we have that
\begin{equation*} 
|FBL(A)^*| \leq |\mathbb{R}^{\hat{A}}| = |(2^{\mathbb{N}})^{\hat{A}}| = |2^{\mathbb{N} \times \hat{A}}| = 2^{|\hat{A}|} = 2^{|A|}.
\end{equation*} 
\end{proof}

\section{Banach lattices on which every lattice homomorphism attains its norm}\label{SectionPositiveCases}

In this section, we give some sufficient conditions so that the set $\Hom(X, \R)$ is a subset of $\NA(X, \R)$, that is, every lattice homomorphism on $X$ attains its norm. From the description given in Section \ref{SectionStructureHom} (see Example \ref{ExampleSetHomClassical}), we have that every lattice homomorphism defined on a classical Banach lattice attains its norm. The examples of items (a) and (c) are all examples of Banach lattices with order continuous norm. The norm of a Banach lattice $X$ is said to be \textit{order continuous} if $\inf \{\|x\|: x\in A\}=0$ whenever $A\subset X$ is a downward directed set such that $\inf(A)=0$. We refer the reader to \cite[Section 2.4]{Meyer} for a detailed background on Banach lattices with order continuous norm. In particular, we will use that a Banach lattice has order continuous norm if and only if every monotone order bounded sequence is convergent (see \cite[Theorem 2.4.2]{Meyer}). Recall that a sequence $\left(x_n \right)_{n\in \N}$ in a Banach lattice $X$ is \textit{order bounded} if there are $x,y \in X$ such that $x \leq x_n \leq y$ for every $n\in \N$. In the next theorem we show that	 every lattice homomorphism on a Banach lattice with order continuous norm attains its norm.	

\begin{thm}
	\label{TheoNAOrderCont}
	Let $X$ be a Banach lattice with order continuous norm. Then, $\Hom(X,\R) \subseteq \NA (X,\R)$.
\end{thm}
\begin{proof}
	Let $x^*\colon X \longrightarrow \R$ be a lattice homomorphism different from zero. Take any sequence $(x_n)_{n \in \N} \subseteq B_X$ such that $x^*(x_n)$ converges to $\|x^*\|$. Changing $x_n$ by $|x_n|$ if it is  necessary, we can assume that each $x_n$ is positive and $x^*(x_n)>0$ for every $n\in \N$. Consider the sequence $\left(y_n\right)_{n \in \N}$ given by the formula
	\begin{equation*} 
	y_n = \bigwedge_{k \leq n} \left( \frac{\|x^*\|}{x^*(x_k)} x_k\right)   \mbox{ for every } n\in \N.
	\end{equation*} 
	Notice that $\left(y_n\right)_{n \in \N}$ is a positive decreasing sequence, so it is order bounded (by $y_1$ and the vector zero). By \cite[Theorem 2.4.2]{Meyer}, it converges to some $y \in X$.
	Moreover, it follows from the monotonicity of the norm that 
	\begin{equation*} 
	\| y \| \leq \left\| \frac{\|x^*\|}{x^*(x_n)} x_n \right\| = \frac{\|x^*\|}{x^*(x_n)} \|x_n\| \leq \frac{\|x^*\|}{x^*(x_n)}
	\end{equation*} 
	for every $n\in \N$. Since $\frac{\|x^*\|}{x^*(x_n)}$ converges to $1$, we conclude that $y \in B_X$. We claim that $|x^*(y)|=\|x^*\|$. Indeed, this is immediate since $$x^*(y_n) = x^* \left( \bigwedge_{k \leq n} \left( \frac{\|x^*\|}{x^*(x_k)} x_k\right)\right) =  \bigwedge_{k \leq n} \left( \frac{\|x^*\|}{x^*(x_k)} x^*(x_k) \right)=\|x^*\| $$
	for every $n\in \N$ and $y$ is the limit of $\left(y_n\right)_{n \in \N}$. Thus, $x^* \in \NA(X, \R)$, as desired.
\end{proof}



%
%
%

\bigskip

A natural class of Banach lattices generalizing the class of Banach lattices with order continuous norm is the class of $\sigma$-Dedekind complete Banach lattices. Recall that a Banach lattice is said to be \textit{$\sigma$-Dedekind complete} if every order bounded sequence in it has a supremum or an infimum.
We do not know whether the previous theorem can be extended to $\sigma$-Dedekind complete Banach lattices. The main difficulty is that, although lattice homomorphisms respect lattice operations, they might not respect infinite suprema and infima, as we can see in the next example.
\begin{example}
Take $K=\N \cup \{\infty\}$ the one point compactification of the natural numbers with the discrete topology. Then, $\delta_\infty \in C(K)^*$ is a lattice homomorphism by Example \ref{ExampleSetHomClassical}. Take $f_n=\chi_{\{1,\ldots,n\}}$ the characteristic function of the set $\{1,\ldots,n\}$. Then, $\left(f_n\right)_{n \in \N}$ is an increasing sequence. Moreover, the supremum $\bigvee_{n \in \N} f_n$ exists and it is the constant function $1$. Nevertheless, 
$$\bigvee_{n \in \N} \delta_\infty (f_n)=0 \neq 1 = \delta_\infty(1)= \delta_\infty \left(\bigvee_{n \in \N} f_n\right) .$$ 
\end{example}

Notice that every $\sigma$-Dedekind complete Banach lattice without order continuous norm contains a subspace isomorphic to $\ell_\infty$ (see \cite[Proposition 1.a.7]{LinTza12}). In particular, every separable $\sigma$-Dedekind complete Banach lattice has order continuous norm. Now, bearing in mind that every dual Banach lattice is $\sigma$-Dedekind complete (see the comment below \cite[Definition 1.a.3]{LinTza12}), we get the following:

\begin{cor}
 $\Hom(X,\R) \subseteq \NA (X,\R)$ whenever $X$ is a separable dual Banach lattice or, in general, a dual Banach lattice not containing a subspace isomorphic to $\ell_\infty$.
\end{cor}

\bigskip

On the other hand, the class of Banach lattices with order continuous norm generalizes the class of {\it Kantorovich-Banach spaces} (KB-space, for short). These are the Banach lattices in which every norm bounded monotone sequence is norm convergent.
This class of Banach lattices coincides with the class of Banach lattices not containing a sublattice isomorphic to $c_0$ (see \cite[Theorem 2.4.12]{Meyer}) or, equivalently, a subspace isomorphic to $c_0$ (see the Remark in page 35 of \cite{LinTza12}). Thus, the class of KB-spaces generalizes in turn the class of reflexive Banach lattices.

\bigskip

The most natural examples of Banach lattices without order continuous norm are $C(K)$-spaces. In order to show that Theorem \ref{TheoNAOrderCont} also holds for this class, we consider the more general class of Banach lattices of the form $C(K,X)$, where $K$ is a compact Hausdorff topological space and $X$ is a Banach lattice. The following characterization follows from Lemma \ref{lemma:HomomC(K,X)}.

\begin{prop}
	Let $X$ be a Banach lattice and $K$ a compact Hausdorff topological space. Then, $\Hom(C(K,X),\R) \subseteq \NA(C(K,X),\R)$ if and only if $\Hom(X,\R) \subseteq \NA(X,\R)$.
\end{prop}

\begin{proof}
Suppose first that there exists $x^* \in \Hom(X,\R)$ not attaining its norm. Take any $a \in K$.
Then, the formula $\varphi_{x^*}(f)=x^*(f(a))$ for every $f \in  C(K,X)$ defines a lattice homomorphism. Since $\varphi_{x^*}(1 \otimes x)= x^*(x)$ for every $x \in X$, it follows that $\|\varphi_{x^*}\|=\|x^*\|$. Moreover, since $f(K) \subset B_X$ for every function $f$, it is immediate that $\varphi_{x^*}$ does not attain its norm.

Now suppose that there is $\varphi \in \Hom(C(K,X),\R)$ which does not attain its norm. By Lemma \ref{lemma:HomomC(K,X)}, there exists $x^* \in  \Hom(X,\R)$ and $a \in K$ such that
$ \varphi(f)=x^*(f(a))$ for every function $f \in  C(K,X)$.
By a similar argument, $\|\varphi\|=\|x^*\|$ and $x^*$ does not attain its norm.
\end{proof}

We summarize the main results obtained in this section in the  following corollary.

\begin{cor} \label{reflexive1} $\Hom(X, \R) \subseteq \NA(X, \R)$  and   $\Hom(C(K,X),\R) \subseteq \NA(C(K,X),\R)$ in the following cases:
\begin{itemize}
\item[(a)] $X$ is a KB-space or, equivalently, $X$ does not contain a subspace isomorphic to $c_0$.
\item[(b)] $X$ is lattice isometric to $c_0(\Gamma)$ for some set $\Gamma$ or, more generally, whenever $X$ has order continuous norm.
\item[(c)] $X$ is a dual lattice not containing $\ell_\infty$.
\end{itemize}	
\end{cor}

\begin{rem}
Notice that the examples of the previous corollary include the case when $X$ is reflexive. For the sake of completeness, we recall the reader that the following assertions for a Banach lattice $X$ are equivalent (see \cite[Theorem 2.4.15 and Proposition 5.4.13]{Meyer} and \cite[Theorem 4.71 and Theorem 5.29]{AB}):
\begin{enumerate}
	\item $X$ is reflexive;
	\item $X$ and $X^*$ are KB-spaces;
	\item $X$ does not contain any subspace isomorphic to $c_0$ or to $\ell_1$;
	\item $X$ does not contain any sublattice isomorphic to $c_0$ or to $\ell_1$;
	\item $X$ and $X^*$ have the Radon-Nikod\'ym property;
	\item $X^{**}$ has the Radon-Nikod\'ym property;
	\item $\ell_1$ is not lattice embeddable in either $X$ or $X^{*}$;
	\item Every positive operator from $\ell_1$ to $X$ is weakly compact.
\end{enumerate}
\end{rem}

\bigskip
 
Let us finish this section by making a simple but interesting remark about a phenomenon that happens to be true in both categories. It is well-known that every compact operator defined on a reflexive Banach space attains its norm and that reflexive Banach spaces are exactly those Banach spaces in which every functional attains its norm. We wonder if the same happens in the Banach lattice setting. Namely, we wonder whether every compact lattice homomorphism $T\colon X \longrightarrow Y$ attains its norm whenever $X$ and $Y$ are Banach lattices such that $\Hom(X, \R) \subseteq \NA(X, \R)$.
We prove that this is the case, at least, when $Y$ is an abstract $M$ space or, equivalently, $Y$ is lattice isometric to a sublattice of a $C(K)$-space (see, for instance, \cite[Theorem 1.b.6]{LinTza12}).

\begin{thm}
\label{TheoCompactNA}	
Let $X$ be a Banach lattice such that $\Hom(X, \R) \subseteq \NA(X, \R)$ and $Y$ an abstract $M$ space. Then, every compact lattice homomorphism $T\colon X \longrightarrow Y$ attains its norm.
\end{thm}
\begin{proof}
Take $\left(x_n\right)_{n \in \N}$ a sequence in $B_X$ such that $\left(\|Tx_n\|\right)_{n \in \N}$ converges to $\|T\|$. Moreover, since $T$ is compact, we can suppose that $\left(Tx_n\right)_{n \in \N}$ is norm convergent to some $y \in Y$. Notice that $\|y\|=\|T\|$.

Now, notice that in any $C(K)$-space we have that, for every $f\in C(K)$, there is $x^* \in \Hom(C(K),\R)$ such that $\|x^*\|=1$ and $x^*(f)=\|f\|$ (just take $x^*$ to be any evaluation functional $\delta_a$ with $a\in K$ an arbitrary point where $f$ attains its maximum). Since $Y$ can be seen as a sublattice of a $C(K)$-space, this property is inherited by $Y$. Thus, there is $y^*\in \Hom(C(K),\R)$ such that $\|y^*\|=1$ and $y^*(y)=\|y\|=\|T\|$. Now, notice that $x^* = y^* \circ T \in \Hom(X, \R) \subseteq \NA(X, \R)$, so there exists $x \in B_X$ such that $x^*(x)=\|x^*\|$. Notice that $x^*(x_n)=y^*(Tx_n)$ converges to $y^*(y)=\|T\|$ and, since $y^* \in S_{Y^*}$, it follows that $\|x^*\|=\|T\|$.  Thus, $x^*(x)=y^*(Tx)=\|T\|$ and $T$ attains its norm at $x$.
\end{proof}
	
The next example shows that, in general, the condition that $Y$ is an abstract $M$ space cannot be dropped.
\begin{example}
\label{ExamCompactNotNA}
Let $T\colon c_0 \longrightarrow \ell_1$ be the compact lattice homomorphism defined by the formula $T(\sum_{n=1}^\infty \lambda_n e_n)= \sum_{n=1}^\infty \frac{\lambda_n}{2^n} e_n$. Then, although $\Hom(c_0, \R) \subseteq \NA(c_0, \R)$ and $T$ is a compact lattice homomorphism, it does not attain its norm. 
\end{example}

\section{Lattice homomorphisms which do not attain their norms}\label{SectionNegativeCases}

Every lattice homomorphism we have been working with so far attains its norm (see Section \ref{SectionPositiveCases} and, in particular, Corollary \ref{reflexive1}). Notice that these cases include all classical Banach lattices. One may wonder whether, in general, \textit{every} lattice homomorphism attains its norm. The class of free Banach lattices, which has arised  during the past few years as an important source of counterexamples when comparing properties of Banach spaces and Banach lattices, provides a suitable setting to answer negatively this question.

\bigskip

Let us recall that if $E$ is a Banach space and $x^* \in E^*$, then the evaluation function $\delta_{x^*} \colon FBL[E] \longrightarrow \R$ is defined by $\delta_{x^*}(f) := f(x^*)$ for every $f \in FBL[E]$.
The aims of this section are threefold. First, we will be interested in answering whether the inclusion $\Hom(X, \R) \subseteq \NA(X, \R)$ holds for an arbitrary Banach lattice $X$. We will see next that this is not the case and we give several concrete examples of Banach spaces $E$ such that there exists a lattice homomorphism on $FBL[E]$ which does not attain its norm (see Corollary \ref{CorollaryMainTheorem}). On the other hand, we try to characterize those lattice homomorphisms on $FBL[E]$ which attain their norm. Namely, we wonder whether $x^* \in \NA(E, \R)$ if and only if $\delta_{x^*} \in \NA(FBL[E],\R)$ holds true.
Finally, as an application, we show that the natural lattice version of the Bishop-Phelps theorem fails in a drastic way.


\bigskip

We start with the main theorem of the section. This will follow by a combination of Proposition \ref{l_1(A)PropertyP}, Theorem \ref{conjepredu}, and Theorem \ref{Theorem1Complemented} below.

\begin{thm}\label{MainTheorem}
If $E$ is a Banach space which contains a 1-complemented copy of
\begin{enumerate}
\item \label{item1MainTheorem} $\ell_1(A)$ for some infinite set $A$, or
\item \label{item2MainTheorem} an isometric predual of $\ell_1(A)$ for some infinite set $A$,
\end{enumerate}
then there exists $x^* \in E^*$ such that $\delta_{x^*} \notin \NA(FBL[E], \R)$.
\end{thm}
In particular, we have the following concrete examples of Banach spaces $E$ such that there exists a lattice homomorphism in $\Hom(FBL[E], \R)$ which does not attain its norm.

\begin{cor}\label{CorollaryMainTheorem}
There exists $x^* \in E^*$ such that $\delta_{x^*} \notin \NA(FBL[E], \R)$ whenever $E$ is
\begin{enumerate}
\item \label{item1CorollaryMainTheorem} an infinite-dimensional $L_1$-space for some measure $\mu$;
\item \label{item2CorollaryMainTheorem} a separable infinite-dimensional isometric predual of an $L_1$-space. In particular, when $E$ is a $C(K)$-space with $K$ metrizable;
\item \label{item3CorollaryMainTheorem} if $E=X\oplus_a Y$ and there exists $x^* \in X^*$ such that $\delta_{x^*} \notin \NA(FBL[X], \R)$, where $\oplus_a$ denotes an arbitrary absolute sum;
\item \label{item4CorollaryMainTheorem} If $E=X\widehat{\otimes}_\alpha Y$ and there exists $x^* \in X^*$ such that $\delta_{x^*} \notin \NA(FBL[X], \R)$, where $\alpha$ is any uniform cross norm in $X\otimes Y$, and there exists $x^* \in X^*$ such that $\delta_{x^*} \notin \NA(FBL[X], \R)$;
\item \label{item5CorollaryMainTheorem} $E$ is the Lipschitz-free space $\mathcal F(M)$ for a complete metric space $M$ such that $M'\neq \emptyset$ or $M$ contains an infinite ultrametric subspace.
\end{enumerate}

\end{cor}

\begin{proof}
(\ref{item1CorollaryMainTheorem}) follows from the fact that $\ell_1$ is 1-complemented in any infinite-dimensional $L_1(\mu)$ space for every $\mu$ (see \cite[Lemma 5.1.1]{AlbKal}).

(\ref{item2CorollaryMainTheorem}) follows from the fact that every separable infinite-dimensional isometric predual of an $L_1$-space contains a $1$-complemented subspace isometric to $c_0$ (see  \cite[Corollary 1.5]{Gasparis}).

%
%

(\ref{item3CorollaryMainTheorem}) follows because, if $E=X\oplus_a Y$ then $X$ is $1$-complemented in $E$ (see e.g. \cite{Hardtke} and references therein for background in absolute sums), so the result follows by Theorem \ref{Theorem1Complemented}.

(\ref{item4CorollaryMainTheorem}) follows because, if $E=X\widehat{\otimes}_\alpha Y$ then $X$ is $1$-complemented in $E$ (see e.g. \cite[Chapter 6]{ryan} and references therein for background on cross norms in tensor products), so the result follows by Theorem \ref{Theorem1Complemented}.

Finally (\ref{item5CorollaryMainTheorem}) follows by \cite[Theorem 1]{CJ} and Theorem \ref{Theorem1Complemented}.\end{proof}

In order to prove Theorem \ref{MainTheorem}, we need some preliminary results. We will be using the following lemma with no explicit reference from now on.

\begin{lem} 
\label{PropNormFree}	
Let $E$ be a Banach space. If $x^* \in E^*$, then $\| \delta_{x^*}\| = \|x^*\|$.
\end{lem}

\begin{proof} If $x^* = 0$, then for every $f \in FBL[E]$, we have that $\delta_{x^*}(f) = f(0) = 0$, and in consequence $\delta_{x^*} = 0$.
	
Let $x^* \not= 0$ and $f \in FBL[E]$. It follows from the definition of the norm in $FBL[E]$ and the fact that every $f\in FBL[E]$ is positively homogeneous that, since $\ds \left(\frac{x^*}{\|x^*\|}\right)(x) \leq 1$ for every $x \in B_E$, 
\begin{equation*}
\|f\|_{FBL[E]} \geq \left| f \left( \frac{x^*}{\|x^*\|} \right) \right| = \frac{1}{\|x^*\|} |f(x^*)|,
\end{equation*}
which implies that $|f(x^*)| \leq \|x^*\| \|f\|_{FBL[E]}$. So,
\begin{equation*}
\| \delta_{x^*}\| = \sup_{f \in B_{FBL[E]}}  |\delta_{x^*}(f)| = \sup_{f \in B_{FBL[E]}} |f(x^*)| \leq \|x^*\|. 
\end{equation*}
On the other hand, for every $x \in B_E$, we have
\begin{equation*}
\| \delta_{x^*}\| \geq | \delta_{x^*} (\delta_x)| = |\delta_x(x^*)| = |x^*(x)|,
\end{equation*}
which implies that $\|\delta_{x^*}\| \geq \|x^*\|$. 	
\end{proof} 



From the proof of Lemma \ref{PropNormFree} together with the James theorem, we can extract the following consequence.

\begin{prop}\label{PropDeltaNonNormAttaining} 
Let $E$ be a Banach space. If $x^*\in \NA(E, \R)$, then $\delta_{x^*} \in \NA(FBL[E], \R)$.
In particular, if $E$ is reflexive, then every lattice homomorphism on $FBL[E]$ attains its norm.
\end{prop}


As we have mentioned before, we are interested in the converse of Proposition \ref{PropDeltaNonNormAttaining}, which we do not know if it holds true for every Banach space $E$. Let us highlight it as a conjecture:

\begin{con}\label{MainConjecture}
Let $E$ be a Banach space. Then, $x^* \in \NA(E, \R)$ if and only if $\delta_{x^*} \in \NA(FBL[E], \R)$.
\end{con}

First, we prove that Conjecture \ref{MainConjecture} is separably determined, that is, if it holds for every separable Banach space, then it does for every Banach space. For the better understanding of the proof of it, we send the reader to the very last part of Section \ref{section:notation}.

\begin{thm}\label{falloconjec}
Let $E$ be a Banach space such that Conjecture \ref{MainConjecture} does not hold. Then, there exists a separable ai-ideal $Z$ in $E$ such that Conjecture \ref{MainConjecture} does not hold.
\end{thm}

\begin{proof} Let us assume that Conjecture \ref{MainConjecture} does not hold. Then, there exists $x^*\notin \NA(E, \mathbb{R})$ with $\|x^*\| = 1$ and $f\in S_{FBL[E]}$ such that $\delta_{x^*}(f) = f(x^*)=1$. We will prove that there exist a separable ai-ideal $Z$ in $E$, $z^* \in Z^*$, and $g \in FBL[Z]$ such that $z^* \not\in \NA(Z, \R)$ and $\delta_{z^*}(g) = \|\delta_{z^*}\|$. This is a combination of Steps 1, 2, and 3 below.

\vspace{0.2cm}	
\noindent	
{\bf Step 1}: There exists a separable ai-ideal $Z$ in $E$ and $z^* \in S_{Z^*}$ such that $z^* \not\in \NA(Z, \R)$.	
\noindent	
\vspace{0.2cm}		
	
By Remark \ref{LemFunctionsDependFinitelyManyCoordinates}, there exists a separable subspace $Y \subseteq E$ such that $f(x^*)=f(y^*)$ whenever $x^*\vert_{Y}=y^*\vert_{Y}$. Now, using the notation of Theorem \ref{theo:exteaiideales}, let us set $W := \{x^*\} \subseteq E^*$. Then, we can find a separable ai-ideal $Z$ in $E$ with $Y\subseteq Z\subseteq E$ and an almost isometric Hahn-Banach extension operator $\varphi\colon Z^*\longrightarrow E^*$ with $x^*\in \varphi(Z^*)$. Therefore, we have that $x^*=\varphi(z^*)$ for some $z^*\in Z^*$. In particular, $\|z^*\| \geq \|\varphi(z^*)\| = \|x^*\| = 1$. On the other hand, since $\varphi(z^*)(z) = z^*(z)$ for every $z \in Z$ and $z^* \in Z^*$, and $\varphi$ is an isometry, for every $z\in B_Z$, we get that
\begin{equation*} 
z^*(z)=\varphi(z^*)(z)=x^*(z)<\Vert x^*\Vert=\Vert z^*\Vert.
\end{equation*}
This gives that $\|z^*\| \leq 1$ and it cannot attain its norm.

\vspace{0.2cm}	
\noindent	
{\bf Step 2}: There exists $g: Z^* \longrightarrow \R$ with $\|g\|_{FBL[Z]} = 1$ such that $\delta_{z^*}(g) = \|\delta_{z^*}\|  = 1$.	
\noindent	
\vspace{0.2cm}

Define $g:=f\circ \varphi\colon Z^*\longrightarrow \mathbb R$. Let us prove that $\Vert g\Vert_{FBL[Z]} = \Vert f\Vert_{FBL[E]} = 1$ and that $g$ attains its norm at $z^*$. Indeed, let $z_1^*,\ldots, z_n^*\in Z^*$ be such that $\sup_{z\in B_Z} \sum_{i=1}^n \vert z_i^*(z)\vert\leq 1$. This is equivalent to the fact that $\Vert \sum_{i=1}^n \xi_i z_i^*\Vert_{Z^*}\leq 1$ holds for every choice of signs $\xi_i\in \{-1,1\}$. Given any choice of signs $\xi_1,\ldots, \xi_n\in \{-1,1\}$, we get that
\begin{equation*} 
\left\| \sum_{i=1}^n \xi_i \varphi(z_i^*) \right\|_{E^*} = \left\| \varphi \left( \sum_{i=1}^n \xi_i z_i^* \right) \right\|_{E^*} \leq \Vert \varphi \Vert \left\| \sum_{i=1}^n \xi_i z_i^* \right\|_{Z^*} \leq 1.
\end{equation*} 
Since $\xi_1, \ldots, \xi_n$ are arbitrary, we deduce that $\sup_{x\in B_E} \sum_{i=1}^n \vert \varphi(z_i^*)(x)\vert\leq 1$. Now,
\begin{equation*} 
\sum_{i=1}^n \vert g(z_i^*)\vert=\sum_{i=1}^n | f(\varphi(z_i^*)| \leq \Vert f\Vert_{FBL[E]}=1.
\end{equation*} 
This proves that $\Vert g\Vert_{FBL[Z]} \leq 1$. On the other hand, 
\begin{equation*} 
\delta_{z^*}(g) = g(z^*)=f(\varphi(z^*))=f(x^*)=1. 
\end{equation*}

\vspace{0.2cm}	
\noindent	
{\bf Step 3}: The function $g$ in Step 2 belongs to $FBL[Z]$.	
\noindent	
\vspace{0.2cm}

 
We will prove that $g$ is in the closed vector lattice generated by the $\delta_z$'s with $z\in Z$. First, let us notice that for every $z\in Z$ and $z^*\in Z^*$, we have that
\begin{equation} \label{varphi} 
(\delta_z\circ \varphi)(z^*)=\varphi(z^*)(z)=z^*(z)=\delta_z (z^*).
\end{equation} 
This means that the function $\phi\colon \mathbb R^{E^*}\longrightarrow \mathbb R^{Z^*}$ given by
\begin{equation*} 
\phi(h)=h\circ \varphi
\end{equation*} 
satisfies that if $\Vert h\Vert_{FBL[E]}<\infty$, then $\Vert \phi(h)\Vert_{FBL[Z]}\leq \Vert h\Vert_{FBL[E]}$. Also, by using (\ref{varphi}), we have that $\phi(\delta_z)=\delta_z \circ \varphi = \delta_z \in FBL[Z]$ holds for every $z\in Z$. Furthermore, by definition,  $\phi$ is linear and preserves suprema and infima. This implies that if $h$ is an element in the vector lattice generated by $\{\delta_z: z\in Z\}$, then $\phi(h)=h\circ\varphi\in FBL[Z]$.

Now, since $f$ depends on the coordinates of $Y \subseteq Z$, we can take a sequence $(f_n)_{n \in \mathbb{N}}$ in the vector lattice generated by $\{\delta_y:y\in Y\}$ such that $f_n\rightarrow f$ in $FBL[E]$. Since, for every $n\in\mathbb N$, $f_n$ is in the vector lattice generated by $\{\delta_z:z\in Z\}$, we get that $\phi(f_n)=f_n\circ \varphi\in FBL[Z]$ holds for every $n\in\mathbb N$. Let us notice that $f_n\circ \varphi$ is a Cauchy sequence in $FBL[Z]$. Indeed, given $n,k\in\mathbb N$, we get that
\begin{equation*} 
\Vert (f_n\circ\varphi)-(f_k\circ\varphi)\Vert_{FBL[Z]}=\Vert(f_n-f_k)\circ\varphi\Vert_{FBL[Z]}\leq \Vert f_n-f_k\Vert_{FBL[E]},
\end{equation*} 
from where the Cauchy condition follows since $(f_n)_{n \in \N} \subseteq FBL[E]$ is Cauchy. By completeness, $f_n\circ\varphi\rightarrow \tilde{g}$ for some $\tilde{g}\in FBL[Z]$. To finish the proof, we prove that $g=\tilde{g}$. To this end, let us see that $g(z^*)=\tilde{g}(z^*)$ holds for every $z^*\in Z^*$. Given $z^*\in Z^*$, we get that
$$\tilde{g}(z^*)=\lim_n (f_n\circ\varphi)(z^*)=\lim_n f_n(\varphi(z^*))=f(\varphi(z^*))=(f\circ\varphi)(z^*),$$
where we have used both that $f_n\circ \varphi\rightarrow \tilde{g}$ in $FBL[Z]$ and that $f_n\rightarrow f$ in $FBL[E]$. Hence, $g=f\circ\varphi=\tilde{g}\in FBL[Z]$, as desired.
\end{proof}

\bigskip

In what follows, we are giving a wide list of Banach spaces that satisfy Conjecture \ref{MainConjecture}. In fact, we are presenting Banach spaces which have the following property.

\begin{defn}\label{DefinitionPropertyP}
A Banach space $E$ has {\it property (P)} if for every $x^* \not\in \NA(E, \R)$, the set 
$$ C:=\{y^* \in E^*: |x^*(x)|+|y^*(x)|\leq \|x^*\| \mbox{ for every } x\in B_E \}$$
satisfies that $x^*$ is in the $w^*$-closure of $\R^+ C := \{ \lambda y^*: \lambda >0,~y^*\in C\}$.
\end{defn}

Although artificial at a first sight, it turns out that Banach spaces with property $(P)$ satisfy Conjecture \ref{MainConjecture}.

\begin{lem}\label{LemmaPropertyP}
Let $E$ be a Banach space with property $(P)$. Then, $x^* \in \NA(E, \R)$ if and only if $\delta_{x^*} \in \NA(FBL[E], \R)$.
\end{lem}

\begin{proof} By Proposition \ref{PropDeltaNonNormAttaining}, we just need to prove that if $x^*\not\in \NA(E, \R)$, then $\delta_{x^*} \not\in \NA(FBL[E], \R)$. Indeed, let $x^* \not\in \NA(E, \R)$. Suppose, without loss of generality, that $\|x^*\| = 1$. Consider the set
\begin{equation*} 
C:=\{y^* \in B_{E^*}: |x^*(x)|+|y^*(x)|\leq \|x^*\| \mbox{ for every } x\in B_E \}.
\end{equation*} 
If $\delta_{x^*}$ is norm-attaining, then there is $f \in FBL[E]$ with $\|f\|_{FBL[E]} = 1$ such that $\delta_{x^*}(f) = f(x^*) = \|\delta_{x^*}\| = \|x^*\| = 1$. Since $f(x^*)=1$, $f$ is $w^*$-continuous on $B_{E^*}$, and $E$ has property $(P)$, there is $y^*\in C$ such that $f(y^*)>0$. Thus, for every $x\in B_E$, $|x^*(x)|+|y^*(x)|\leq 1$ and it follows from the definition of the norm $\|\cdot\|_{FBL[E]}$ that
	\begin{equation*}
		\|f\|_{FBL[E]} \geq |f(y^*)| + |f(x^*)| > |f(x^*)| = 1,
	\end{equation*}	
	which is a contradiction.	
\end{proof}
 
Let us remark that property $(P)$ is {\it not} satisfied by every Banach space $E$. Indeed, every separable Banach space can be endowed with an equivalent norm which makes it strictly convex (see the proof of \cite[Theorem 8.13]{Fabian}). 
Moreover, if $E^*$ is strictly convex, then every point of the sphere is an extreme point. If $E$ is not reflexive, then there are points on the sphere of $E^*$ which are extreme but do not attain its norm. Now, notice that $C \neq \{0\} $ for a point $x^*$ in the sphere if and only if $x^*$ is not an extreme point.

\bigskip

On the other hand, we have some Banach spaces $E$ satisfying property $(P)$.  

\begin{prop}\label{l_1(A)PropertyP}
Let $A$ be an infinite set. Then, $\ell_1(A)$ has property $(P)$. In particular, $x^* \in \NA(\ell_1(A), \R)$ if and only if $\delta_{x^*} \in \NA(FBL[\ell_1(A)], \R)$.
\end{prop}

\begin{proof}
Let $x^*\in \ell_\infty(A)=\ell_1(A)^*$ with $\Vert x^*\Vert=1$. Suppose that it does not attain its norm. Let us prove that, given any finite set $F\subset A$, we can find $y^*\in \ell_\infty(A)$ and $\lambda>0$ such that $\Vert x^*\pm y^*\Vert\leq 1$ and $\lambda y^*(t)=x^*(t)$ holds for every $t\in F$. This is enough in view of the $w^*$-topology on $\ell_\infty(A)$.  To this end, let $F\subset A$ be an arbitrary finite set. Since $x^*$ does not attain its norm, we have that $\sup_{t\in F}\vert x^*(t)\vert=\alpha<1$. Now, define $y^*\in \ell_\infty(A)$ by 
$$y^*(t):=\left\{\begin{array}{cc}
(1-\alpha)x^*(t) & \mbox{if }t\in F;\\
0 & \mbox{ otherwise.}
\end{array} \right.$$
Let us prove that $\Vert x^*\pm y^*\Vert=\sup\limits_{t\in A} \vert x^*(t)\pm y^*(t)\vert\leq 1$. For this, we consider two cases.
\begin{enumerate}
\item If $t\notin F$, then we get that $y^*(t) = 0$ and so
$$\vert x^*(t)\pm y^*(t)\vert=\vert x^*(t)\vert < \Vert x^*\Vert=1.$$
\item If $t\in F$, then we get that $\vert x^*(t)\vert\leq \alpha$ and so
$$\vert x^*(t)\pm y^*(t)\vert\leq \alpha+(1-\alpha)\vert x^*(t)\vert\leq 1.$$
\end{enumerate}
Hence, taking supremum in $A$, we have $\Vert x^*\pm y^*\Vert\leq 1$. Finally, taking $\lambda:=\frac{1}{1-\alpha}$, for any $t\in F$ we have that $\lambda y^*(t)=\lambda (1-\alpha)x^*(t)=x^*(t)$, as desired.
\end{proof}

Isometric preduals of $\ell_1(A)$, for $A$ an infinite set, satisfy Conjecture \ref{MainConjecture}. Indeed, this is a consequence of the fact that isometric preduals of $\ell_1$ have property $(P)$, as we can see in the following theorem.

\begin{thm}\label{conjepredu}
Let $E$ be an isometric predual of $\ell_1(\Gamma)$ for some infinite set $\Gamma$. Then an element $x^* \in \NA(E, \R)$ if and only if $\delta_{x^*} \in \NA(FBL[E], \R)$ (in other words, $E$ satisfies Conjecture \ref{MainConjecture}).
\end{thm}

\begin{proof}
We show first that, by Theorem \ref{falloconjec}, we can suppose that $\Gamma$ is a countable set, so $E$ is an isometric predual of $\ell_1$. Indeed, suppose that $Z$ is a separable ai-ideal in $E$. By the proof of \cite[Theorem 1]{rao2}, $Z$ is a separable isometric predual of $L_1(\mu)$ for some measure $(\Omega,\Sigma,\mu)$. Moreover, $Z$ is Asplund since it is a subspace of an Asplund space (note that $E^*$ has the Radon-Nikod\'ym property). Thus, $Z^*$ is separable and has the Radon-Nikod\'ym property. This implies that $Z^* = \ell_1$; in other words, $\mu$ is purely atomic. Indeed, assume by contradiction that there exists some subset $A$ with $0<\mu(A)$ and such that $\mu_{|A}$ does not contain any atom. Then the mapping
$$\begin{array}{ccc}
L_1(\mu)& \longrightarrow & L_1(\mu_{|A})\oplus_1 L_1(\mu_{|\Omega\setminus A})\\
f & \longmapsto & (f\chi_A,f\chi_{\Omega\setminus A})
\end{array} $$
is an onto linear isometry, so $L_1(\mu)$ contains an isometric copy of $L_1(\mu_{|A})$, and $L_1(\mu_{|A})$ fails the Radon-Nikod\'ym property because it is easy to see that its unit ball does not have any extreme point, which entails a contradiction with the fact that $Z^*=L_1(\mu)$ has the Radon-Nikod\'ym property. This contradiction proves that $Z^*=\ell_1$.

Let $(e_n^*)_{n=1}^{\infty}$ be the Schauder basis of $E^*$ isometrically equivalent to the usual $\ell_1$-basis, i.e. 
\begin{equation*} 
\left\| \sum_{i=1}^n a_i e_i^* \right\| = \sum_{i=1}^n |a_i| 
\end{equation*} 
for every $n \in \mathbb{N}$ and scalar sequences $(a_i)_{i=1}^n$. Denote by $F_n$ the closed span of $\{ e_1^*, \ldots, e_n^*\}$. Suppose that $x^* \in E^*$ does not attain its norm. Without loss of generality, we may suppose that $\|x^*\| = 1$. Then, $x^*$ has infinite support. Indeed, this follows from the following claim which may have its own interest.

\bigskip

\noindent	
{\bf Claim}: Let $E$ be a Banach space. Suppose that $E^*$ is isometric to $\ell_1(\N)$. If $x^* \in E^*$ is finitely supported, then $x^* \in \NA(E, \R)$.
\noindent	
\vspace{0.2cm}

By \cite[Corollary 4.1]{Gasparis}, there exists a $w^*$-continuous contractive projection $Q_n$ from $E^*$ onto $F_n$ such that $E_n := Q_n^* F_n^*$ is isometric to $\ell_{\infty}^n$ and $\bigcup_{n=1}^{\infty} E_n$ is dense in $E$. Suppose that $x^* \in E^*$ is finitely supported. So, for some $n \in \N$, we have that $x^* = \sum_{j=1}^n a_j e_j^* \in F_n$ with $a_j \in \R$ for $j=1,\ldots, n$. This implies that $Q_n(x^*) = x^*$. Given $x \in E$, if $J \colon E \longrightarrow E^{**}$ denotes the embedding of $E$ into $E^{**}$, we have that
	\begin{eqnarray*}
		x^*(x) = J(x)(x^*) = J(x)(Q_n(x^*)) &=& (J(x) \circ Q_n)(x^*) \\
		&=& (Q_n^* \circ J(x))(x^*) \\
		&=& x^*(Q_n^*J(x)).
	\end{eqnarray*}
	This shows that, for a fixed $x \in E$, the action of $x^*$ at $x$ is the same as the action of $x^*$ at $Q_n^*J(x)$.
	Since $\|Q_n\| = 1$ and $B_{Q_n^*F_n^*} = B_{E_n}$ is compact, we conclude that $x^*$ must attain its norm and this proves the claim.\\

Now, set $x^* = \sum_{j=1}^{\infty} a_j e_j^*$ with $a_j \in \R$ for every $j \in \N$. We will prove that $x^*$ is in the $w^*$-closure of the set $\R^+ C = \{\lambda y^*: \lambda > 0, y^* \in C\}$, where $C$ is the set defined in Definition \ref{DefinitionPropertyP}. In order to do this, for each $n \in \N$ we construct elements $y^*\in C$ as follows.

If $\sum_{k \leq n} |a_k| = 0$, we just take $y^* = 0 \in C$. Suppose now that $\sum_{k \leq n} |a_k| > 0$. Since $x^*$ has infinite support, there is $m > n$ such that $a_m \not= 0$. Since $x^* \in E^*$ and
	\begin{equation*}
	1 = \|x^*\| = \left\| \sum_{j=1}^{\infty} a_j e_j^* \right\| = \sum_{j=1}^{\infty} |a_j|,
	\end{equation*}
	we can pick $m$ big enough so that
	\begin{equation*}
	|a_m| = \lambda \sum_{k \leq n} |a_k|
	\end{equation*}	
	for some $\lambda \in (0, 1)$. Set
	\begin{equation*}
	y^* := \lambda \sum_{k \leq n}  a_k e_k^* - a_m e_m^* \in E^*.
	\end{equation*}
	Then, 
	\begin{eqnarray*}
		\|x^* + y^*\| &=& \left\| \sum_{k=1}^{\infty} a_k e_k^* + \sum_{k \leq n} \lambda a_k e_k^* - a_m e_m^* \right\| \\
		&=&\left\| \lambda \sum_{k \leq n} a_k e_k^* + \sum_{k\not=m} a_k e_k^* \right\| \\
		&\leq& |a_m| + \sum_{k \not= m} |a_k| \\
		&=& \sum_{k=1}^{\infty} |a_k| = \|x^*\| = 1,	
	\end{eqnarray*}
	and
	\begin{eqnarray*}
		\|x^* - y^*\| &=& \left\| \sum_{k=1}^{\infty} a_k e_k^* - \sum_{k \leq n} \lambda a_k e_k^* + a_m e_m^* \right\| \\
		&=& \left\| (1 - \lambda) \sum_{k \leq n} a_k e_k^* + \sum_{\substack{k>n\\k\not= m}} a_k e_k^* +2 a_m e_m^* \right\| \\
		&=& (1 - \lambda) \sum_{k \leq n} |a_k| + \sum_{\substack{k>n\\k\not= m}} |a_k| + 2 |a_m| \\
		&=& \sum_{k \not= m} |a_k| - \lambda \sum_{k\leq n} |a_k| + 2 \lambda \sum_{k \leq n} |a_k| \\
		&=& \sum_{k \not= m} |a_k| + \lambda \sum_{k \leq n} |a_k| \\
		&=& \sum_{k \not= m} |a_k| + |a_m| = \sum_{k=1}^{\infty} |a_k| = \|x^*\| \leq 1.
	\end{eqnarray*} 
	This implies that $y^* \in C$. 
	
	Let us end by proving, using the element $y^*$ as defined above, that $x^*$ is in the $w^*$-closure of the set $\mathbb R^+ C$. To this end, pick a $w^*$-open set 
	\begin{equation*} 
	W:=\{z^*\in E^*: \vert x^*(x_i)-z^*(x_i)\vert<\varepsilon\ \mbox{for} \ 1\leq i\leq k\}
	\end{equation*} 
	for certain $x_1,\ldots, x_k\in E$. Since $\bigcup_{n\in\mathbb N} E_n$ is dense in $E$, we can assume that $x_i\in E_n$ for a large enough $n\in\mathbb N$ and for every $i\in\{1,\ldots, k\}$. Pick $y^*\in C$ and $0<\lambda<1$ such that $Q_n(x^*)=\frac{1}{\lambda}Q_n(y^*)$ as constructed before. Given $i\in\{1,\ldots, k\}$, we have that, since $x_i\in E_n=Q_n^*F_n^*$, then $x_i=Q_n^*(x_i)$. Hence,
	\begin{eqnarray*} 
		\frac{1}{\lambda}y^*(x_i) = \frac{1}{\lambda}y^*(Q_n^*(x_i)) &=& \frac{1}{\lambda}Q_n^*(x_i)(y^*) \\
		&=&\frac{1}{\lambda}(J(x_i)\circ Q_n)(y^*) \\ 
		&=& J(x_i) \left(\frac{1}{\lambda}Q_n (y^*) \right)\\
		&=& J(x_i) (Q_n (x^*))\\
		&=& Q_n^*(x_i)(x^*) \\
		&=& x^*(Q_n^*(x_i))\\
		&=& x^*(x_i).
	\end{eqnarray*} 
	Since $i$ was arbitrary, we get that $\frac{1}{\lambda}y^*\in W$, so we are done.
\end{proof}

In order to be completely ready to prove Theorem \ref{MainTheorem}, we need a last result. 

\begin{thm}\label{Theorem1Complemented}
Let $E$ be a Banach space and assume that $F$ is a $1$-complemented subspace of $E$. Suppose that there exists $y^*\in F^*$ such that $\delta_{y^*} \not\in \NA(FBL[F], \R)$. Then, there exists $x^*\in E^*$ such that $\delta_{x^*} \not\in \NA(FBL[E], \R)$.
\end{thm}

\begin{proof} Let $y^* \in F^*$ such that $\delta_{y^*}$ does not attain its norm. Assume, without loss of generality, that $\Vert y^*\Vert=1$. Consider a norm-one projection $P\colon E\longrightarrow F$. Let us define $x^*:=P^*(y^*) \in E^*$. Notice that $x^* \in S_{E^*}$ since $\|y^*\| = 1$ and $F$ is 1-complemented. 
	
\vspace{0.2cm}	
\noindent	
{\bf Claim}: $\delta_{x^*}$ does not attain its norm.	
\noindent	
\vspace{0.2cm}

On the contrary, let us assume that $\delta_{x^*}$ attains its norm. Then, there exists $f\in S_{FBL[E]}$ such that $f(x^*)=f(P^*(y^*))=(f\circ P^*)(y^*)=1$. We will show that $f\circ P^*\in FBL[F]$ and $\Vert f\circ P^* \Vert_{FBL[F]}\leq 1$, which will imply that $\delta_{y^*}$ attains its norm and this will give the desired contradiction.

We prove first that $f\circ P^*$ has finite norm on $FBL[F]$. Indeed, let $y_1^*,\ldots, y_k^*\in F^*$ be such that $\sup\limits_{y\in B_F}\sum_{i=1}^k \vert y_i^*(y)\vert\leq 1$. Then,
\begin{equation*}
\sup_{x \in B_E} \sum_{i=1}^k |(P^*y_i^*)(x))| = \sup_{x \in B_E} \sum_{i=1}^k |y_i^*(P(x))| = \sup_{y \in B_F} \sum_{i=1}^k |y_i^*(y)| \leq 1.
\end{equation*}
Hence
\begin{equation*} 
\sum_{i=1}^k \vert (f\circ P^*)(y_i^*)\vert=\sum_{i=1}^k \vert f(P^*y_i^*)\vert\leq \Vert f\Vert_{FBL[E]}=1,
\end{equation*} 
which proves that $f\circ P^*\in \mathbb R^{Y^*}$ has finite norm and that it is smaller than or equal to $1$. On the other hand, $(f\circ P^*)(y^*)=f(x^*)=1$ by assumption.

Let us finally prove that $f\circ P^*\in FBL[F]$. To this end, take $(f_n)_{n \in \mathbb{N}}$ a sequence depending on finitely many coordinates in $FBL[E]$ such that $f_n\rightarrow f$. Notice that $f_n\circ P^*\in FBL[F]$. Indeed, given any $x\in E$ and $y^*\in F^*$ it follows that
$$(\delta_x\circ P^*)(y^*)=\delta_x(P^*(y^*))=P^*(y^*)(x)=y^*(P(x))=\delta_{P(x)}(y^*),$$
which means that $\delta_{x}\circ P^*=\delta_{P(x)}$. This proves that $f_n\circ P^*\in FBL[F]$ holds for every $n\in\mathbb N$. Now, an argument involving Cauchy condition on the sequence $(f_n)_{n \in \mathbb{N}}$ similar to the one in Theorem \ref{falloconjec} implies that $f\circ P^*\in FBL[F]$.
\end{proof}

\bigskip

We finish the paper by showing that there is no Bishop-Phelps type theorem for lattice homomorphisms. Recall that the Bishop-Phelps theorem states that the set of norm-attaining functionals in a dual Banach space is norm-dense. In the Banach lattice setting the situation is extremely opposite; we cannot approximate any not norm-attaining lattice homomorphism by
norm-attaining lattice homomorphisms.

\begin{thm}\label{TheoremNoBishop-Phelps}
	Let $X$ be a Banach lattice and $x^* \in \Hom(X,\R)$ a lattice homomorphism in $S_{X^*}$ which does not attain its norm. Then, $\|x^*-y^*\| \geq 1$ for any $y^*  \in \Hom(X, \R) \cap \NA(X,\R)$.
\end{thm}
\begin{proof}
Since $y^*$ attains its norm whereas $x^*$ does not, both lattice homomorphisms are linearly independent. By Corollary \ref{CorDistanceHomomorphisms}, we have $\|x^*-y^*\| \geq \|x^*\|=1$.
\end{proof}

\bigskip

In conclusion, we have seen that on several free Banach lattices there exist lattice homomorphism which do not attain their norm. As far as we know, these are the first examples of not norm-attaining lattice homomorphisms in the literature. We wonder if the existence of a lattice homomorphism which does not attain its norm on a Banach lattice $X$ implies that $X$ contains some kind of free structure. In particular, we wonder if $X$ contains an isomorphic copy of a free Banach lattice $FBL[E]$ whenever there exists $x^* \in \Hom(X,\R)$ which does not attain its norm.

\vspace{0.2cm}
\noindent
\textbf{Acknowledgements:} We are grateful to Miguel Mart\'in and Pedro Tradacete for fruitful conversations on the topic of the paper. Moreover, we would like to thank Antonio Avil\'es for pointing out several important ideas regarding Theorem \ref{TheoremDisjointFamiliesInTheBanachSpaceDualOfFBL(A)}.


\begin{thebibliography}{99}
	
\bibitem {abrahamsen} \textsc{T.~A. Abrahamsen}, \textit{Linear extensions, almost isometries, and diameter two}, Extracta Math. 30, 2 (2015), 135--151.

\bibitem {aln2}  \textsc{T.~A. Abrahamsen, V. Lima and O. Nygaard}, \textit{Almost isometric ideals in Banach spaces}, Glasgow Math. J. 56 (2014), 395--407.

\bibitem{AA} \textsc{Y~.A.~Abramovich and C.~D. ~Aliprantis}, \textit{An Invitation to Operator Theory},  Graduate Studies in Mathematics, 2002.

\bibitem {gasp} \textsc{M.~D.~Acosta, F.~J.~Aguirre and R.~Pay\'a}, \textit{There is no bilinear Bishop-Phelps theorem}, Isr. J. Math. 93 (1996), 221--227.

\bibitem{AGM} \textsc{M.~D.~Acosta, D.~Garc\'ia and M.~Maestre}, \textit{A multilinear Lindenstrauss theorem}, J. Func. Anal. 235 (2006), 122--136.

\bibitem{AlbKal} \textsc{F. Albiac and N. J. Kalton}, \textit{Topics in Banach Space Theory}, Graduate Texts in Mathematics 233, Springer, 2016.

\bibitem{AB} \textsc{C.~D.~Aliprantis and O.~Burkinshaw}, \textit{Positive Operators}, Handbook of the Geometry of Banach Spaces, Springer, 2006.

\bibitem{AFW} \textsc{R.~M.~Aron, C.~Finet and E.~Werner}, \textit{Some remarks on norm-attaining $n$-linear forms}, Function Spaces (K. Jarosz, ed.), Lecture Notes in Pure and Appl. Math. 172 Marcel Dekker, New York, (1995), 19--28.

\bibitem{ADM} \textsc{R.~M.~Aron, D.~Garc\'ia and M.~Maestre}, \textit{On norm attaining polynomials}, Publ. Res. Inst. Math. Sci. 39 (2003), 165--172.


\bibitem{APRA} \textsc{A. Avil\'es, G. Plebanek and J. D. Rodr\'{i}guez Abell\'{a}n},  \textit{Chain conditions in free Banach lattices}, J. Math. Anal. Appl. 465 (2018), 1223--1229.
 
\bibitem{ART18} \textsc{A. Avil\'es, J. Rodr\'iguez and P. Tradacete}, \textit{The free Banach lattice generated by a Banach space}, J. Funct. Anal. 274 (2018), 2955--2977.

\bibitem {BP} \textsc{E.~Bishop and R.~R.~Phelps}, \textit{A proof that every Banach space is subreflexive}, Bull. Am. Math. Soc. 67 (1961), 97--98.

\bibitem{B} \textsc{J.~Bourgain}, \textit{On dentability and the Bishop-Phelps property}, Israel J. Math. 28 (1977), 265--271.

\bibitem{CRX} \textsc{J.~Cao, I.~Reilly and H.~Xiong}, \textit{A Lattice-valued Banach-Stone Theorem}, Acta
Math. Hungarica 98 (2003), 103--110.

\bibitem {CCGMR} \textsc{B.~Cascales, R.~Chiclana, L.~C.~Garc\'ia-Lirola, M.~Mart\'in and A.~Rueda Zoca}, \textit{On strongly norm attaining Lipschitz maps}, J. Funct. Anal. 277 (2019), 1677--1717.

\bibitem{MPR2} \textsc{R.~Chiclana and M.~Mart\'in}, \textit{The Bishop-Phelps Bollob\'as property for Lipschitz maps}, Nonlinear Analysis 188 (2019), 158--178.

\bibitem {CGMR} \textsc{R.~Chiclana, L.~C.~Garc\'ia-Lirola, M.~Mart\'in and A.~Rueda Zoca}, \textit{Examples and applications of the density of strongly norm attaining Lipschitz maps}, (2020). To appear in Rev. Mat. Iberoam.

\bibitem{Choi} \textsc{Y.~S.~Choi}, \textit{Norm attaining bilinear forms on $L_1[0,1]$}, J. Math. Anal. Appl. 211 (1997), 295-300.

\bibitem{CJ} \textsc{M.~C\'uth and M. Johanis}, \textit{Isometric embedding of $\ell_1$ into Lipschitz-free spaces and $\ell_{\infty}$ into their duals}, Proc. Amer. Math. Soc. 145 (2017), 3409--3421.

\bibitem{dPW15} \textsc{B.\ de Pagter,  A. W.\  Wickstead}, \textit{Free and projective Banach lattices}, Proc. Royal Soc. Edinburgh Sect. A 145 (2015), 105--143.

\bibitem{Fabian} \textsc{M.~Fabian, P.~Habala, P.~H\'ajek, V.~Montesinos, J.~Pelant and V.~Zizler}, \textit{Functional Analysis and Infinite-Dimensional Geometry}, Springer-Verlag, 2001.

\bibitem{FHHMZ} \textsc{M.~Fabian, P.~Habala, P.~H\'ajek, V.~Montesinos and V.~Zizler}, \textit{Banach space theory. The basis for linear and nonlinear analysis}, CMS Books in Mathematics/Ouvrages de Math\'ematiques de la SMC. Springer-Verlag, New York, 2010.

\bibitem{Gasparis} \textsc{I. Gasparis}, \textit{On contractively complemented subspaces of separable $L_1$-preduals}, Israel J. Math. 128 (2002), 77--92.

\bibitem{gks} \textsc{G.~Godefroy, N.~J.~Kalton and P.~D. Saphar}, \textit{Unconditional ideals in Banach spaces}, Studia Math. 104 (1993), 13--59.

\bibitem{Hardtke} \textsc{J.~Hardtke},\textit{ Absolute sums of Banach spaces and some geometric properties related to rotundity and smoothness}, Banach J. Math. Anal. 8 (2014), 295--334.

\bibitem{H} \textsc{R.~E.~Huff}, \textit{Dentability and the Radon-Nikod\'ym property}, Duke Math. J. 41, (1974) 111--114.

\bibitem{J} \textsc{R.~C.~James}, \textit{Characterizations of reflexivity}, Studia Math. 23 (1964), 205--216.

\bibitem{JW} \textsc{J.~Johnson and J.~Wolfe}, \textit{Norm attaining operators}, Studia Math. 65 (1979), 7--19.

\bibitem{kaltonlocomp} \textsc{N.~J.~Kalton}, \textit{Locally Complemented Subspaces and $L_p$-spaces for $0 <p < 1$}, Math. Nachr. 115 (1984), 71--97.

\bibitem{LLOT} \textsc{D.~Leung, L.~Li, T. Oikhberg and M.~A.~Tursi}, \textit{Separable universal Banach lattices}, Israel J. of Math. 231 (2018), 141--152.

\bibitem {linds2} \textsc{J.~Lindenstrauss}, \textit{On operators which attain their norm}, Isr. J. Math. 1 (1963), 139--148.

\bibitem{LinTza12} \textsc{J. Lindenstrauss and L. Tzafriri}, \textit{Classical Banach Spaces II}, Springer-Verlag, 1979.

\bibitem{Meyer} \textsc{P.~Meyer-Nieberg}, \textit{Banach lattices}, Springer-Verlag, 1991.

\bibitem{OT} \textsc{T.~Oikhberg and M.~A.~Tursi}, \textit{Order extreme points and solid convex hulls}, preprint, https://arxiv.org/abs/1907.00660.

\bibitem {rao2} \textsc{T.S.S.R.K Rao}, \textit{On ideals and generalized centers of finite sets in Banach spaces}, J. Math. Anal. Appl. 398 (2013), 886-888.

\bibitem {ryan} \textsc{R.~A.~Ryan}, \textit{Introduction to tensor products of Banach spaces}, Springer Monographs in Mathematics, Springer-Verlag, London, 2002.

\bibitem{S} \textsc{W.~Schachermayer}, \textit{Norm attaining operators on some classical Banach spaces}, Pacific J. Math. 105 (1983), 427--438.

\bibitem {SY} \textsc{B.~Sims and D.~Yost}, \textit{Linear Hahn-Banach extension operators}, Proc. Edin. Math. Soc. 32 (1989), 53--57.



\bibitem{U} \textsc{J.~J.~Uhl}, \textit{Norm attaining operators on $L_1[0,1]$ and the Radon-Nikod\'ym property}, Pacific J. Math. 63 (1976), 293-300.

\bibitem{Z} \textsc{V.~Zizler}, \textit{On some extremal problems in Banach spaces}, Math. Scand. 32 (1973), 214--224.



\end{thebibliography}
\end{document}